\documentclass[12pt]{amsart}

\usepackage{amssymb,verbatim,graphicx}
\usepackage[mathscr]{eucal}
\usepackage{enumerate}
\usepackage{xspace}

\newtheorem{theorem}{Theorem}[section]

\newtheorem{proposition}[theorem]{Proposition}
\newtheorem{corollary}[theorem]{Corollary}

\newtheorem{definition}[theorem]{Definition}

\theoremstyle{definition}
\newtheorem{example}[theorem]{Example}
\newtheorem{remark}[theorem]{Remark}

\numberwithin{equation}{section}

\def\&{\wedge}
\newcommand{\alp}{\alpha}

\newcommand{\kap}{\kappa}

\newcommand{\A}{\mathbb{A}}
\newcommand{\E}{\mathbb{E}}
\newcommand{\R}{\mathbb{R}}

\newcommand{\e}{\mathbf{e}}
\newcommand{\bfv}{\mathbf{v}}
\newcommand{\bv}{\mathbf{v}}
\newcommand{\bfx}{\mathbf{x}}

\newcommand{\bb}{\mathbb}

\begin{document}

\title{A Tale of Two Arc Lengths: Metric notions for curves in surfaces in equiaffine space}

\author{Jeanne N. Clelland}
\address{Department of Mathematics, 395 UCB, University of
Colorado,
Boulder, CO 80309-0395}
\email{Jeanne.Clelland@colorado.edu}
\author{Edward Estrada}
\address{Department of Physics, 390 UCB, University of Colorado, Boulder, CO, 80309-0390}
\email{Edward.Estrada@colorado.edu}
\author{Molly May}
\address{Department of Physics, 390 UCB, University of Colorado, Boulder, CO, 80309-0390}
\email{Molly.May@colorado.edu}
\author{Jonah Miller}
\address{Department of Physics, 390 UCB, University of Colorado, Boulder, CO, 80309-0390}
\email{Jonah.Miller@colorado.edu}
\author{Sean Peneyra}
\address{Department of Physics, 390 UCB, University of Colorado, Boulder, CO, 80309-0390}
\email{peneyra.s@gmail.com}
\author{Michael Schmidt}
\address{Department of Physics, 390 UCB, University of Colorado, Boulder, CO, 80309-0390}
\email{Michael.Schmidt@colorado.edu}

% These are the AMS classification numbers that seemed most appropriate to me; most journals want some of these.

% We should acknowledge my NSF grant:
\thanks{This research was supported in part by NSF grant DMS-0908456.}

\begin{abstract}
In Euclidean geometry, all metric notions (arc length for curves, the first fundamental form for surfaces, etc.) are derived from the Euclidean inner product on tangent vectors, and this inner product is preserved by the full symmetry group of Euclidean space (translations, rotations, and reflections).  
In equiaffine geometry, there is no invariant notion of inner product on tangent vectors that is preserved by the full equiaffine symmetry group.  Nevertheless, it is possible to define an invariant notion of arc length for nondegenerate curves, and an invariant first fundamental form for nondegenerate surfaces in equiaffine space.  This leads to two possible notions of arc length for a curve contained in a surface, and these two arc length functions do not necessarily agree.  In this paper we will derive necessary and sufficient conditions under which the two arc length functions do agree, and illustrate with examples.
\end{abstract}

\maketitle

\section{Introduction}\label{intro-sec}

% Here's where we tell you what we're going to do and why it's super cool.  And how it's related to \cite{SomeOldPaper}.

The primary defining characteristic of Euclidian geometry in $\R^3$ is the presence of a flat metric $\langle , \rangle$ which is defined on all tangent vectors to all points of $\R^3$ and invariant under the action of the Euclidean group. When studying submanifolds of the Euclidean space $\E^3$ (i.e., $\R^3$ together with a Euclidean metric), all metric properties (e.g., arc lengths, surface areas) are derived from this underlying metric.  By contrast, in equiaffine geometry (which, for convenience, we will refer to simply as ``affine geometry"), it is not possible to define a metric on tangent vectors which is preserved by the action of the equiaffine group.  There is an invariant volume form, but no invariant notion of distance which can be restricted to submanifolds of $\A^3$ (i.e., $\R^3$ together with an equiaffine structure) in any obvious way.

Nevertheless, it is possible to define a notion of affine arc length for generic curves in affine space, as well as a notion of an affine metric for generic surfaces, in such a way that these notions are preserved by the action of the equiaffine group.  Because there is no inner product on tangent vectors, these affine notions of metrics on submanifolds depend on higher-order derivatives, as opposed to the analogous Euclidean notions, which depend only on first derivatives of the submanifolds in question.

Now, suppose that we have a curve $\alpha$ contained in a surface $\Sigma \subset \A^3$.  The affine metric on $\Sigma$ can be restricted to $\alpha$ in order to define an arc length function on $\alpha$.  Unlike in Euclidean geometry, it is possible for this arc length function to differ from the affine arc length function on $\alpha$ considered as a curve in $\A^3$.  The goal of this paper is to explore these two different notions of affine arc length for $\alpha$ and to consider conditions under which they may agree.

The remainder of the paper is structured as follows.  In \S \ref{preliminaries-sec} we will recall the definitions of affine arc length for nondegenerate curves in $\A^3$ and the affine first fundamental form for nondegenerate surfaces in $\A^3$.  In \S \ref{comparison-sec} we will explore how these two notions give rise to two different notions of arc length for a curve $\alpha \subset \Sigma \subset \A^3$, and we will construct several examples where these notions do not agree. In \S \ref{theorem-sec}, we will determine conditions under which these two notions do agree, and in \S \ref{examples-sec}, we will construct examples of curves on surfaces for which the two notions of arc length are the same.

\section{Notions from affine geometry}\label{preliminaries-sec}

\subsection{Curves and affine arc length}

Let $I \subset \R$, and let $\alp:I \to \R^3$ be a regular curve.  ($\alp$ may be considered as a curve in either $\E^3$ or $\A^3$.)  In Euclidean geometry, one generally associates to $\alp$ the {\em Frenet frame} $(\e_1, \e_2, \e_2)$ defined by:
\[
\e_1(t) =  \frac{\alpha'(t)}{\|\alpha'(t)\|}, \qquad
\e_2(t) =  \frac{\e_1'(t)}{\|\e_1'(t)\|}, \qquad
\e_3(t) = \e_1(t) \times \e_2(t),
\]
where for any vector $\bv$, we define $\| \bv \| = \sqrt{\langle \bv, \bv \rangle}$.
(See, e.g., \cite{doCarmo76}.) This frame is well-defined provided that the vectors $\alpha'(t), \alpha''(t)$ are linearly independent for each $t \in I$; such a curve is called a {\em nondegenerate} curve in $\E^3$.  The {\em Euclidean arc length} function $\bar{s}(t)$ of $\alpha$ is defined by
\[ \bar{s}(t) = \int_0^t \sqrt{\langle \alp'(\sigma), \alp'(\sigma) \rangle} \, d\sigma. \]
It has the property that if $\alpha$ is reparametrized via the inverse function $t(\bar{s})$ as
\[ \alpha(\bar{s}) = \alpha(t(\bar{s})), \]
then $\|\alpha'(\bar{s})\| \equiv 1$, and so the Frenet frame satisfies $\e_1(\bar{s}) = \alpha'(\bar{s})$.  Moreover, the Frenet frame satisfies the {\em Frenet equations}
\[ \begin{bmatrix} \e_1'(\bar{s}) &  \e_2'(\bar{s}) & \e_3'(\bar{s}) \end{bmatrix} = \begin{bmatrix} \e_1(\bar{s}) & \e_2(\bar{s}) & \e_3(\bar{s}) \end{bmatrix} \begin{bmatrix} 0 & -\kappa(\bar{s}) & 0 \\[0.1in] \kappa(\bar{s}) & 0 & \tau(\bar{s}) \\[0.1in] 0 & -\tau(\bar{s}) & 0 \end{bmatrix}, \]
where $\kappa(\bar{s}), \tau(\bar{s})$ are the curvature and torsion functions, respectively, of $\alpha$.
An important observation is that the matrix $[ \e_1(\bar{s}) \ \ \e_2(\bar{s}) \ \ \e_3(\bar{s}) ]$, whose columns are the Frenet frame vectors, is an element of the Lie group $SO(3)$, which is precisely the (oriented) symmetry group of the Euclidean structure on each tangent space.

Now suppose that we consider $\alpha:I \to \R^3$ as a curve in $\A^3$.  We no longer have the notions of vector norm and cross product to make use of in order to define a frame along the curve.  But if we define
\[
\e_1(t) =  \alpha'(t), \qquad
\e_2(t) =  \alpha''(t), \qquad
\e_3(t) =  \alpha'''(t),
\]
then the quantity
\[ \det [ \e_1(t) \ \ \e_2(t) \ \ \e_3(t) ] \]
is invariant under the action of the equiaffine group.  By way of analogy with the Euclidean case, it would be nice to find a parametrization of the curve for which the matrix $[ \e_1(t) \ \ \e_2(t) \ \ \e_3(t) ]$ is an element of the symmetry group $SL(3)$; i.e., for which
\[ \det [ \e_1(t) \ \ \e_2(t) \ \ \e_3(t) ] = 1. \]

This motivation leads to the following definitions (see, e.g., \cite{Su83}):
\begin{itemize}
\item A curve $\alpha:I \to \A^3$ is called {\em nondegenerate} if the vectors $\alpha'(t), \alpha''(t), \alpha'''(t)$ are linearly independent for each $t \in I$.  (Note that this is different from the Euclidean definition.)
\item The {\em affine arc length} function $s_{\alpha}(t)$ of a nondegenerate curve $\alpha$ is defined by
\begin{equation}
s_{\alp}(t) = \int_{0}^t \sqrt[6]{\det[\alp'(\sigma) \ \ \alp''(\sigma) \ \ 
\alp'''(\sigma)]}\, d\sigma .
\label{eq:s:alpha}
\end{equation}
(Note that this assumes that $\det [ \e_1(t) \ \ \e_2(t) \ \ \e_3(t) ]  >0$; if $\det [ \e_1(t) \ \ \e_2(t) \ \ \e_3(t) ]  < 0$, then $\alp(t)$ can be reparametrized as $\alp(-t)$ to reverse the sign.  In particular, note that this assumption yields a preferred orientation for $\alp$.)
\end{itemize}
If $\alp$ is reparametrized by the inverse function $t(s_{\alp})$ as
\[ \alp(s_{\alp}) = \alp(t(s_{\alp})), \]
then the vectors 
\[
\e_1(s_{\alp}) =  \alpha'(s_{\alp}), \qquad
\e_2(s_{\alp}) =  \alpha''(s_{\alp}), \qquad
\e_3(s_{\alp}) =  \alpha'''(s_{\alp})
\]
form the {\em affine Frenet frame} of $\alp$.  The affine Frenet frame has the property that the matrix $[\e_1(s_{\alp}) \ \ \e_2(s_{\alp}) \ \ \e_3(s_{\alp}) ]$ is an element of $SL(3)$, and the affine analog of the Frenet equations is
\[ \begin{bmatrix} \e_1'(s_{\alp}) & \e_2'(s_{\alp}) & \e_3'(s_{\alp}) \end{bmatrix} = \begin{bmatrix} \e_1(s_{\alp}) & \e_2(s_{\alp}) & \e_3(s_{\alp}) \end{bmatrix} \begin{bmatrix} 0 & 0 & \kappa_1(s_{\alp}) \\[0.1in] 1 & 0 & \kappa_2(s_{\alp}) \\[0.1in] 0 & 1 & 0 \end{bmatrix}, \]
where $\kappa_1(s_{\alp}), \kappa_2(s_{\alp})$ are the {\em affine curvature} functions of $\alp$.  The affine arc length, affine Frenet frame, and affine curvatures are all invariant under the action of the equiaffine group.

Affine arc length is a very different notion from Euclidean arc length.  Some of the differences include:
\begin{itemize}
\item Unlike Euclidean arc length, which 
depends only on the first derivative of $\alp$, the affine arc length 
depends on the first three derivatives of $\alp$.  In general, this number is dependent on the dimension of the ambient affine space: the affine arc length of a curve $\alp: I \to \bb{A}^n$ depends on the first $n$ derivatives of $\alp$.
\item The affine arc length is only nonzero for nondegenerate curves; so for instance, any curve contained in a plane in $\bb{A}^3$ has affine arc length zero according to this definition.  It may, however, have nonzero affine arc length when regarded as a curve in $\bb{A}^2$.  
\end{itemize}
It turns out that the affine arc length function of $\alp$ can be expressed in terms of the Euclidean invariants of $\alp$:

\begin{proposition}\label{affine-arc-length-prop}
Let $\alp:I \to \R^3$ be a nondegenerate curve. Let $\bar{s}(t)$, $\kappa(t), \tau(t)$ be the Euclidean arc length, curvature, and torsion of $\alpha$, respectively, and suppose that $\tau(t)>0$. Then the affine arc length $s_{\alp}(t)$ of $\alp$ is given by
\[ s_{\alp}(t) = \int_0^t \sqrt[6] {\kap(\sigma)^2\, \tau(\sigma)}\, \bar{s}'(\sigma)\, d\sigma. \]
\end{proposition}
In particular, if $\alp$ is parametrized by its Euclidean arc length $\bar{s}$, then
\[ s_{\alp}(\bar{s}) = \int_0^{\bar{s}} \sqrt[6] {\kap(\sigma)^2\, \tau(\sigma)}\, d\sigma. \] 

\begin{proof}
Let $(\e_1(t), \e_2(t), \e_3(t))$ denote the Euclidean Frenet frame of $\alp$, and for convenience, let $v(t)$ denote $\bar{s}'(t)$.  Then we have:
\begin{align*}
\alp'(t) & = v(t) \e_1(t), \\
\alp''(t) & = v'(t) \e_1(t) + v(t)^2 \kappa(t) \e_2(t), \\
\alp'''(t) & = \left(v''(t) - v(t)^3 \kappa(t)^2\right) \e_1(t) + \left(3 v(t) v'(t) \kappa(t) + v(t)^2 \kappa'(t)\right) \e_2(t) \\
& \qquad  + v(t)^3 \kappa(t) \tau(t) \e_3(t).
\end{align*}
Therefore,
\begin{align*}
 \sqrt[6]{\det[\alp'(t) \ \ \alp''(t) \ \ \alp'''(t)]} & = \sqrt[6]{v(t)^6 \kappa(t)^2 \tau(t) \det[\e_1(t) \ \ \e_2(t) \ \ \e_3(t)]} \\
 &  = v(t) \sqrt[6]{\kappa(t)^2\, \tau(t)} \\
 & =  \sqrt[6]{\kappa(t)^2\, \tau(t)} \bar{s}'(t).
\end{align*}
The result follows.
\end{proof}

Although the Euclidean quantities $\kap(t), \tau(t), \bar{s}(t)$ are not invariant under the action of the equiaffine group, this proposition yields the following corollary:

\begin{corollary}
The Euclidean 1-form $ds_{\alp} = \sqrt[6]{\kappa^2 \tau}\, d\bar{s}$ associated to a nondegenerate curve $\alp:I \to \R^3$ is invariant under the action of the equiaffine group.
\end{corollary}

\subsection{Surfaces and the affine first fundamental form}

Let $U \subset \R^2$, and let $X:U \to \R^3$ be a parametrization of a regular surface $\Sigma$ (considered as a surface in either $\E^3$ or $\A^3$).  In Euclidean geometry, one associates to $X$ the {\em first fundamental form}
\[ \text{I}_{\text{Euc}} = E \, du^2 + 2F\, du\, dv + G \, dv^2, \]
where
\begin{equation}
 E = \langle X_u, X_u \rangle, \qquad F = \langle X_u, X_v \rangle, \qquad G = \langle X_v, X_v \rangle. \label{define-EFG-eqn}
\end{equation} 
The first fundamental form expresses the restriction of the Euclidean metric to $\Sigma$ as follows: for any tangent vector $\bfv$ to $\Sigma$, if we express $\bfv$ as 
\[ \bfv = a X_u + b X_v, \]
then
\[ \langle \bfv, \bfv \rangle = \text{I}_{\text{Euc}}(\bfv) = E a^2 + 2Fab + G b^2. \]

Next, one associates to $\Sigma$ the {\em second fundamental form}
\[ \text{II}_{\text{Euc}} = e \, du^2 + 2f\, du\, dv + g \, dv^2, \]
where
\begin{equation}
 e = \langle X_{uu}, N \rangle, \qquad f = \langle X_{uv}, N \rangle, \qquad g = \langle X_{vv}, N \rangle, \label{define-efg-eqn}
\end{equation}
and $N$ is a unit normal vector field to $\Sigma$.  The second fundamental form encapsulates the curvature properties of the surface; in particular, the Gauss curvature of the surface is
\begin{equation} K = \frac{\det (\text{II}_{\text{Euc}})}{\det (\text{I}_{\text{Euc}} )} = \frac{eg-f^2}{EG-F^2}. \label{define-K-eqn}
\end{equation}

Now suppose that we consider $\Sigma$ as a surface in $\A^3$.  There is no obvious analog to the Euclidean first fundamental form which is invariant under the action of the equiaffine group.  However, we can construct a quadratic form which closely approximates the Euclidean second fundamental form: if we set
\[ \ell = \det [ X_u \ \ X_v \ \ X_{uu} ],  \qquad m = \det [ X_u \ \ X_v \ \ X_{uv} ], \qquad   n = \det [ X_u \ \ X_v \ \ X_{vv} ], \]
then the quadratic form
\[ \ell\, du^2 + 2m \, du\, dv + n \, dv^2 \]
on the {\em parametrized} surface $X$ is invariant under the action of the equiaffine group.  However, it is not {\em quite} invariant under a change of parametrization for $\Sigma$: if we set
\[ \bar{X}(\bar{u}, \bar{v}) = X(u(\bar{u},\bar{v}), v(\bar{u},\bar{v})), \]
then we have
\[ \bar{\ell}\, d\bar{u}^2 + \bar{m} \, d\bar{u}\, d\bar{v} + \bar{n} \, d\bar{v}^2 = (\ell\, du^2 + 2m \, du\, dv + n \, dv^2) J, \]
where $J$ is the determinant of the Jacobian matrix of the coordinate transformation
\[ (\bar{u},\bar{v}) \to (u(\bar{u},\bar{v}), v(\bar{u},\bar{v})). \]
This indeterminacy can be remedied as follows: it is straightforward to compute that
\[ \bar{\ell}\bar{n} - \bar{m}^2 = (\ell n-m^2) J^4. \]
Therefore, if $\ell n - m^2 \neq 0$, then the quadratic form
\[ \text{I}_{\text{aff}} = |\ell n - m^2|^{-1/4}(\ell\, du^2 + 2m \, du\, dv + n \, dv^2) \]
is a well-defined, invariant quadratic form on $\Sigma$.  This quadratic form is called the {\em affine first fundamental form} of $\Sigma$, and it can be used to define a metric on the surface $\Sigma$.  Unlike in the Euclidean case, this metric is not necessarily positive definite; it may be positive or negative definite, or indefinite.

We make the following definitions (see \cite{Su83}):
\begin{itemize}
\item A surface $\Sigma$ with parametrization $X:U \to \A^3$ is called {\em nondegenerate} if the quadratic form $\ell\, du^2 + 2m \, du\, dv + n \, dv^2$ is nondegenerate (i.e., if $\ell n-m^2 \neq 0$).
\item A nondegenerate parametrized surface is called {\em elliptic} if the quadratic form $\text{I}_{\text{aff}}$ is definite and {\em hyperbolic} if $\text{I}_{\text{aff}}$ is indefinite.
\end{itemize}
Note that if $\text{I}_{\text{aff}}$ is negative definite, it can be made positive definite by interchanging the roles of $u$ and $v$.  Thus, we will assume without loss of generality that $ \text{I}_{\text{aff}}$ is positive definite in the elliptic case.

As in the case of affine arc length, the affine first fundamental form can be expressed in terms of the Euclidean invariants of $\Sigma$:

\begin{proposition}\label{affine-first-fundamental-form-prop}
Let $X:U \to \R^3$ be a regular parametrization for a nondegenerate surface $\Sigma$.  Let $\text{\em{II}}_{\text{\em{Euc}}}$ denote the Euclidean second fundamental form of $\Sigma$, and let $K$ denote the Euclidean Gauss curvature of $\Sigma$. Then
\begin{equation}
\text{\em{I}}_{\text{\em{aff}}} = |K|^{-1/4} \text{\em{II}}_{\text{\em{Euc}}}.
\label{eq:Iaff:KIIeuc}
\end{equation}
\end{proposition}

\begin{proof}
Let $N$ be the Euclidean normal vector field to $\Sigma$. Then equations \eqref{define-EFG-eqn}, \eqref{define-efg-eqn} and standard properties of determinants imply that
\begin{align*}
\ell & = \det [ X_u \ \ X_v \ \ X_{uu} ] = \det[X_u \ \ X_v \ \ eN] = e \sqrt{EG-F^2} \\
m & = \det [ X_u \ \ X_v \ \ X_{uv} ] = \det[X_u \ \ X_v \ \ fN] = f \sqrt{EG-F^2} \\
n & = \det [ X_u \ \ X_v \ \ X_{vv} ] = \det[X_u \ \ X_v \ \ gN] = g \sqrt{EG-F^2}.
\end{align*}
Therefore, by equation \eqref{define-K-eqn},
\begin{align*}
\text{I}_{\text{aff}} & = |\ell n - m^2|^{-1/4}(\ell\, du^2 + 2m \, du\, dv + n \, dv^2) \\
& = |eg-f^2|^{-1/4} |EG-F^2|^{1/4} (e\, du^2 + 2f \, du\, dv + g \, dv^2) \\
& = |K|^{-1/4} \text{II}_{\text{Euc}}.
\end{align*}
\end{proof}

\begin{corollary}
The Euclidean quadratic form $\text{\em{I}}_{\text{\em{aff}}} = |K|^{-1/4} \text{\em{II}}_{\text{\em{Euc}}}$ associated to a nondegenerate surface $\Sigma \subset \R^3$ is invariant under the action of the equiaffine group.
\end{corollary}

\section{Two arc length functions for curves in surfaces}\label{comparison-sec}

Now suppose that $\alpha:I\to \A^3$ is a curve whose image is contained in a nondegenerate surface $\Sigma = X(U) \subset \A^3$.  The restriction of $\text{I}_{\text{aff}}$ to $\alpha$ defines an arc length function $s_\Sigma$ along $\alpha$, as follows:
\begin{equation}
s_\Sigma(t) = \int_0^t \sqrt{\text{I}_{\text{aff}}(\alpha'(\sigma))}\, d\sigma.
\label{eq:s:sigma}
\end{equation}
We will refer to the function $s_\Sigma$ on $\alpha$ as the {\em induced arc length} function from $\Sigma$.
Although the affine arc length $s_\alpha$ and the induced arc length $s_\Sigma$ are both ``metric functions'' along $\alpha$, in the sense that they allow us to measure the length of a curve embedded in an affine surface, they may or may not agree, even for fairly trivial examples.

\begin{example}\label{sphere-example-1}
Let $\Sigma \subset \R^3$ be the unit sphere $S^2$, with the parametrization
$$
X(u,v) = \left[\cos(u) \cos(v), \sin(u)\cos(v), \sin(v)\right]. 
$$
Regarded as a surface in $\E^3$, $\Sigma$ has uniform Gauss curvature $K=1$ and second fundamental form 
$$
\text{II}_{\text{Euc}} = \cos^2(v) du^2 + dv^2 = \text{I}_{\text{Euc}}.
$$
(In fact, $\Sigma$ is the unique surface in $\E^3$ with the property that
$\text{II}_{\text{Euc}} = \text{I}_{\text{Euc}}$.)
Therefore, the affine first fundamental form of $\Sigma$ is 
\begin{eqnarray}
\text{I}_{\text{aff}} &=& |K^{-1/4}|\, \text{II}_{\text{Euc}}\nonumber\\
&=& \cos^2(v) du^2 + dv^2 \nonumber \\
&=& \text{I}_{\text{Euc}}. \nonumber
\end{eqnarray}
Let $\alpha$ be a great circle on $\Sigma$, parametrized as 
$$\alpha(t) = X(t, 0) = [\cos(t), \sin(t), 0].$$ 
%The affine first fundamental form along $\alpha$ reduces to
%\begin{eqnarray}
%\text{I}_{\text{aff}}(\alpha'(t))  &=& \cos^2(v(t)) \left(\frac{du}{dt}\right)^2 + \left(\frac{dv}{dt}\right)^2\nonumber\\
%&=& \cos^2(0) \left(\frac{d}{dt}(\pi t)\right)^2 + \left(\frac{d}{dt}(0)\right)^2\nonumber\\
%&=& \pi^2.\nonumber
%\end{eqnarray}
%This, in turn, generates the induced arc length function
%\begin{eqnarray}
%s_\Sigma(t) &=& \int_0^t \sqrt{\pi^2}\, d\sigma\nonumber\\
%&=& \pi t.\nonumber
%\end{eqnarray}
Because $\Sigma$ has the property that $\text{II}_{\text{Euc}} = \text{I}_{\text{Euc}}$, the induced arc length function $s_\Sigma(t)$ agrees with the Euclidean arc length function $\bar{s}(t)$; therefore,
\[ s_\Sigma(t) = t. \]
But because $\alpha$ is contained in a plane, it is considered a degenerate curve in $\A^3$ and its affine arc length function $s_\alpha(t)$ is identically equal to zero.
%\begin{figure}[htb]
%\begin{center}
%\leavevmode
%\includegraphics[width=75mm]{Comparison_Pics/sphere_w_great_circle.pdf}
%\includegraphics[width=75mm]{Comparison_Pics/sphere_w_great_circle_arclengths.pdf}\quad
%\caption{(A) A great circle on a sphere is planar. (B) Consequently, the arc length function $s_\alpha(t)$ is identically zero, while $s_\Sigma$ is a nonconstant linear function.}
%\label{fig:great:circle:comp}
%\end{center}
%\end{figure}

For a less trivial example where the two arc lengths do not agree, consider the ``spherical helix" curve
\[ \alpha(t) = X(8t,t) = [\cos(8t) \cos(t), \sin(8t) \cos(t), \sin(t)] \]
in $\Sigma$.  Again, the induced arc length function $s_\Sigma(t)$ is equal to the Euclidean arc length function
\[ s_\Sigma(t) = \int_0^t \sqrt{1 + 64 \cos^2(\sigma)}\,d\sigma, \]
while the affine arc length function $s_\alpha(t)$ is
\[ s_\alpha(t) = \int_0^t \sqrt[6]{48\cos(\sigma)(43 + 672 \cos^2(\sigma))}\, d\sigma. \]
The curve $\alp$ along with graphs of the two arc length functions are shows in Figure \ref{spherical-helix-fig}.
\begin{figure}[htb]
\begin{center}
\leavevmode
\includegraphics[width=2in]{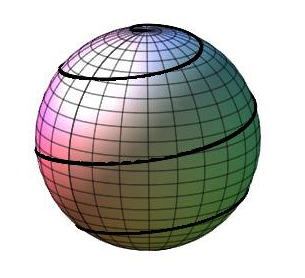}\hspace{0.5in}
\includegraphics[width=2in]{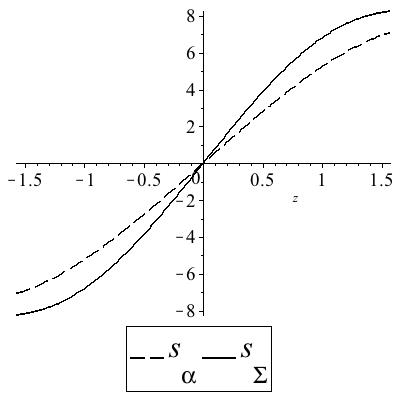}
\caption{The spherical helix of Example \ref{sphere-example-1}}
\label{spherical-helix-fig}
\end{center}
\end{figure}

\end{example}

\begin{example}\label{helicoid-example-1}
Let $\Sigma \subset \R^3$ be the helicoid $H$, with the parametrization
$$X(u,v) = \left[u\cos(v), u\sin(v), v\right].$$
Regarded as a surface in $\E^3$, $\Sigma$ has Gauss curvature $K = -\frac{1}{(u^2+1)^2}$ and second fundamental form
\[ \text{II}_{\text{Euc}} = -\frac{2}{\sqrt{u^2+1}}\, du\, dv. \]
Therefore, the affine first fundamental form of $\Sigma$ is
\[
\text{I}_{\text{aff}} = |K^{-1/4}|\, \text{II}_{\text{Euc}}
= -2\, du\, dv. \]
Let $\alp$ be one of the rulings on $\Sigma$, parametrized as 
\[ \alpha(t) = X(t,v_0) = \left[t \cos(v_0), t \sin(v_0), v_0\right], \]
where $v_0$ is a constant. $\alpha(t)$ is a straight line; thus the vectors $\alpha'$, $\alpha''$, $\alpha'''$ cannot be linearly independent, and the affine arc length $s_\alpha(t)$ is zero.  Moreover, because $\alp$ is an asymptotic curve in $\Sigma$, the restriction of the affine first fundamental form $\text{I}_{\text{aff}}$ to $\alpha'(t)$ is zero. Thus,
$$s_\alpha(t) = s_\Sigma(t) = 0,$$
and the two (degenerate) arc length functions on $\alpha$ coincide.

For a nondegenerate example where the two arc lengths do not agree, consider the ``helical spiral" curve 
\[ \alpha(t) = X(t, \pi t) = \left[t \cos(\pi t), t \sin(\pi t), \pi t\right] \]
in $H$. 
The induced arc length function is
\[ s_{\Sigma}(t) = \int_0^t \sqrt{2\pi}\, d\sigma = \sqrt{2\pi}\, t, \]
while the affine arc length function is
\[ s_{\alpha}(t) = \int_0^t (6\pi^4+\sigma^2\pi^6)^{1/6}\, d\sigma. \]
These functions are qualitatively quite different: $s_\Sigma$ is clearly linear in $t$, while $s_\alpha \sim \pi t^{4/3}.$  The curve $\alp$ along with graphs of the two arc length functions are shows in Figure \ref{helicoid-spiral-fig}.
\begin{figure}[htb]
\begin{center}
\leavevmode
\includegraphics[width=2.3in]{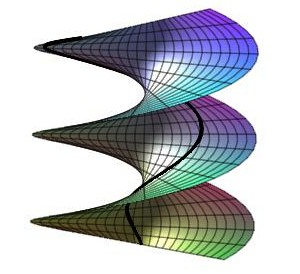}\hspace{0.5in}
\includegraphics[width=2in]{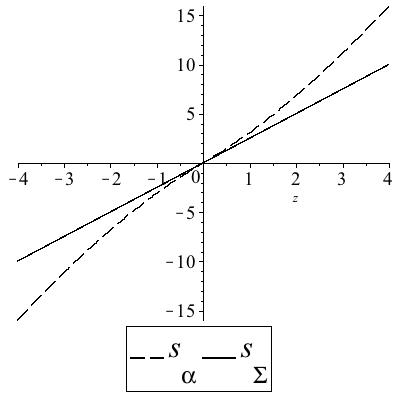}
\caption{The helical spiral of Example \ref{helicoid-example-1}}
\label{helicoid-spiral-fig}
\end{center}
\end{figure}

\end{example}

These examples raise the question: given a surface $\Sigma \subset \A^3$, are there nondegenerate curves in $\Sigma$ for which the two arc length functions $s_\alpha(t), s_\Sigma(t)$ coincide?  We will consider this question in the following section.

\section{Main Theorem}\label{theorem-sec}

In this section, we explore the question of when the affine arc length function $s_\alpha$ and the induced arc length function $s_\Sigma$ for a curve $\alpha \subset \Sigma \subset \A^3$ are equal. This question motivates the following definition.

\begin{definition}
A nondegenerate curve $\alpha$ contained in a regular, nondegenerate surface $\Sigma \subset \A^3$ will be called {\em commensurate} if the affine arc length function $s_\alpha$ and the induced arc length function $s_\Sigma$ associated to $\alpha$ are equal.
\end{definition}

We have seen in \S \ref{comparison-sec} that there exist examples of both commensurate and non-commensurate curves.

\begin{theorem}
Let $X:U\to\R^3$ be a regular parametrization for a nondegenerate surface $\Sigma$, and let $\alpha:I\to\R^3$ be a regular, nondegenerate curve contained in $\Sigma$. Then $\alpha$ is a commensurate curve if and only if, for all $t\in I$,
\begin{equation} %in case we want to make this an equation later
\det\left[\alpha'(t) \ \ \alpha''(t) \ \ \alpha'''(t)\right] = \left[ \text{\em{I}}_{\text{\em{aff}}}(\alpha'(t)) \right]^3. \label{big-theorem-eqn}
\end{equation}
\label{theorem:main:pretty}
\end{theorem}

\begin{proof}
From equations \ref{eq:s:sigma} and \ref{eq:Iaff:KIIeuc}, we know that, along $\alpha$, 
$$ 
s_\Sigma(t) = \int_0^t \sqrt{\text{I}_{\text{aff}}(\alpha'(\sigma))}d\sigma 
.
$$
Similarly, from equation \ref{eq:s:alpha}, we know that
$$
s_\alpha(t) = \int_0^t\sqrt[6]{\det\left[\alpha'(\sigma)\ \ \alpha''(\sigma) \ \ \alpha'''(\sigma)\right]} d\sigma.
$$
The curve $\alpha$ is commensurate if and only if, for all $t \in I$, $s_\alpha(t) = s_\Sigma(t)$; i.e., if and only if
\[  \int_0^t \sqrt[6]{\det\left[\alpha'(\sigma) \ \ \alpha''(\sigma) \ \ \alpha'''(\sigma)\right]} 	
	d\sigma
	= \int_0^t \sqrt{\text{I}_{\text{aff}}(\alpha'(\sigma))}d\sigma . 	
\]
This equation holds for all $t$ if and only if the integrands are equal, i.e., if
\begin{eqnarray}
\sqrt[6]{\det\left[\alpha'(t) \ \ \alpha''(t) \ \ \alpha'''(t)\right]}
	&=& \sqrt{\text{I}_{\text{aff}}(\alpha'(t))}\nonumber\\
\Leftrightarrow \qquad  \det\left[\alpha'(t) \ \ \alpha''(t) \ \ \alpha'''(t)\right]
	&=& \left[\text{I}_{\text{aff}}(\alpha'(t)) \right]^3\nonumber
\end{eqnarray}
Thus, $\alpha$ is commensurate if and only if 
$$\det\left[\alpha'(t) \ \ \alpha''(t) \ \ \alpha'''(t)\right] = \left[\text{I}_{\text{aff}}(\alpha'(t)) \right]^3.$$
\end{proof}

\begin{remark}
Expressing the condition of the theorem by equation \eqref{big-theorem-eqn} allows us to relax the assumption that both sides must be positive.  This is useful when the quadratic form $\text{I}_{\text{aff}}$ is indefinite and the right-hand side may take on negative values.
\end{remark}

We can express the condition \eqref{big-theorem-eqn} in terms of the Euclidean invariants of $\alpha$ and $\Sigma$:

\begin{corollary}\label{commensurate-curve-cor}
Let $X:U \to \R^3$ be a regular parametrization for a nondegenerate surface $\Sigma$, and let $\alpha:I\to\R^3$ be a regular, nondegenerate curve contained in $\Sigma$. Then $\alpha$ is a commensurate curve if and only if, for all $t\in I$, 
\[ \kappa(t)^2 \tau(t) = \left( | K(t) |^{-1/4} k_n(t)\right)^3, \]
where $\kappa(t), \tau(t)$ are the Euclidean curvature and torsion functions of $\alpha$, $K(t)$ is the Gauss curvature of $\Sigma$ at the point $\alpha(t)$, and $k_n(t)$ is the normal curvature of $\Sigma$ at the point $\alpha(t)$ in the direction of $\alpha'(t)$.

\end{corollary}

\begin{proof}
By Proposition \ref{affine-arc-length-prop}, the left-hand side of \eqref{big-theorem-eqn} is equal to 
\[ \kappa(t)^2 \tau(t) \| \alpha'(t) \|^6. \]
By Proposition \ref{affine-first-fundamental-form-prop} and the fact that normal curvature is defined by
\[ k_n(t) = \frac{1}{\|\alpha'(t)\|^2}\text{II}_{\text{Euc}}(\alpha'(t)), \]
the right-hand side of \eqref{big-theorem-eqn} is equal to
\[ \left(|K(t)|^{-1/4} \text{II}_{\text{Euc}}(\alp'(t))\right)^3 =  \left(|K(t)|^{-1/4} k_n(t) \| \alpha'(t) \|^2\right)^3 =  \left(|K(t)|^{-1/4} k_n(t) \right)^3\| \alpha'(t) \|^6.\]
Since $\alp$ is a assumed to be regular, and hence $\|\alp'(t)\| \neq 0$, it follows that
\[ \kappa(t)^2 \tau(t) = \left( | K(t) |^{-1/4} k_n(t)\right)^3. \]
\end{proof}

Theorem \ref{theorem:main:pretty} not only gives us a condition on when a curve is commensurate; it also guarantees the existence of commensurate curves on any nondegenerate surface $\Sigma$.

\begin{corollary}
Let $X:U \to \R^3$ be a regular parametrization for a nondegenerate surface $\Sigma$.  Given any point $\bfx \in \Sigma$ and any tangent vector $\bfv \in T_{\bfx}\Sigma$ for which $\text{{\em I}}_{\text{{\em aff}}}(\bfv) \neq 0$, there exists a 1-parameter family of commensurate curves $\alpha$ in $\Sigma$ such that $\alpha(0) = \bfx$ and $\alpha'(0) = \bfv$.
\label{cor:existence}
\end{corollary}

\begin{remark}
The hypothesis $\text{I}_{\text{aff}}(\bfv) \neq 0$ is not strictly necessary if we extend our definitions to degenerate curves in the obvious way.
\end{remark}

\begin{proof}
Let $\bfx = X(u_0, v_0)$ and $\bfv = a X_u + b X_v$.
We can write 
\begin{equation}
 \alpha(t) = X(u(t), v(t)) \label{local-param-assumption}
\end{equation}
for some smooth functions $u(t), v(t)$. The condition \eqref{big-theorem-eqn} is invariant under reparametrizations of $\alpha$, so without loss of generality we may assume (locally) that $u(t) = u_0 + at$, and therefore
\[ \alpha(t) = X(u_0 + at, v(t)). \]
Equation \eqref{big-theorem-eqn} then becomes a 3rd-order nonlinear ODE for the function $v(t)$. The conditions $\alpha(0) = \bfx$, $\alpha'(0) = \bfv$ are equivalent to the initial conditions
\[ v(0) = v_0, \qquad v'(0) = b \]
for the function $v(t)$.
The local existence/uniqueness theorem for ODEs guarantees that for any real number $c$, there exists a unique local solution to \eqref{big-theorem-eqn} with
\[ v(0) = v_0, \qquad v'(0) = b, \qquad v''(0) = c. \]
\end{proof}

For example, if we set $u_0 = 0, a=1$, so that
\[ \alpha(t) = X(t, v(t)), \]
then equation \eqref{big-theorem-eqn} becomes:
\begin{equation}
\det\left[\frac{d}{dt}X(t,v(t)) \ \ \frac{d^2}{dt^2}X(t,v(t)) \ \ \frac{d^3}{dt^3}X(t,v(t))\right]
	= 
	\left[\text{I}_{\text{aff}}\left(\frac{d}{dt}X(t,v(t))\right)\right]^3 . \label{big-equation-for-v}
\end{equation}

For most surfaces $X$, equation \eqref{big-equation-for-v} is highly nonlinear and cannot be solved explicitly for $v(t)$.  
However, we can often numerically solve for commensurate curves and examine them qualitatively.
In \S \ref{examples-sec}, we will use this process to compute examples of commensurate curves on various surfaces.

\section{Examples}\label{examples-sec}

\begin{comment}
In this section, we proceed to investigate a few cases, including the two previously mentioned, of when there is equality of the two differing arc length integrands as in Theorem (X.X). Using a Rosenbrock-Stiff numerical equation solving algorithm, as well as custom plotting software, implemented through {\sc Maple 15}, we are able to generate families of solutions to the $3^{rd}$ order ODE put forth from Theorem (X.X), these are the aforementioned $commensurate$ curves. Upon successful attainment of these commensurate curves, we observe the appearance of the curves on their generating affine surfaces, and proceed to discuss qualitative behavior of the commensurate curves. To observe behavioral change over a range of initial conditions, we trace several of the solutions in the family of commensurate curves on the parent surface. Such investigation lends some intuition during consideration of when the induced affine arc length, $s_\alpha$ and euclidean arc length, $s_\Sigma$ differ, and in the behavior of $commensurate$ curves. We will consider the following surfaces in affine space: the two dimensional sphere, a hyperboloid, and a paraboloid. For each, the stiff nature of the differential equation generating the commensurate curves leads to severe instabilities at small distances from the initial values during the numerical solution, and thus, a only a limited set of the solution space can be realized/visualized. 
\end{comment}

\begin{example}\label{cool-curves-on-sphere-ex}
Let $\Sigma \subset \R^3$ be the unit sphere, and let $\alpha$ be a commensurate curve on $\Sigma$.  For simplicity, assume that $\alpha$ is parametrized by its Euclidean arc length $\bar{s}$. 
Since all normal curvatures on $\Sigma$ are equal to 1, Corollary \ref{commensurate-curve-cor} implies that the curve $\alpha(\bar{s})$ on $\Sigma$ is commensurate if and only if its curvature and torsion satisfy
\begin{equation}
 \kappa(\bar{s})^2 \tau(\bar{s}) \equiv 1. \label{sphere-example-eqn}
\end{equation}
Moreover, the fact that $\alpha$ lies on the unit sphere implies that
\[ \left( \frac{1}{\kappa(\bar{s})} \right)^2 + \left(\frac{1}{\tau(\bar{s})}\, \frac{d}{ds} \left( \frac{1}{\kappa(\bar{s})} \right)   \right)^2 = 1. \]
(See Exercise 1.3.24 of \cite{Oprea07}.)
Together, these two equations imply that $\kappa(\bar{s})$ satisfies the ODE
\[ (\kappa'(\bar{s}))^2 = \frac{\kappa(\bar{s})^2 - 1}{\kappa(\bar{s})^2}. \]
The general solution of this equation is
\[ \kappa(\bar{s}) = \pm \sqrt{(\bar{s} + c)^2 + 1}, \]
where $c \in \R$.
Since $\kappa(\bar{s})$ is assumed to be positive and $\bar{s}$ is only well-defined up to an additive constant, we may assume without loss of generality that
\[ \kappa(\bar{s}) = \sqrt{\bar{s}^2 + 1}, \]
and then equation \eqref{sphere-example-eqn} implies that 
\[ \tau(\bar{s}) = \frac{1}{\bar{s}^2 + 1}. \]
Unfortunately, the corresponding Frenet equations cannot be integrated analytically, but we can integrate them numerically to obtain the curve shown in Figure \ref{sphere-curves-fig}.  Every other commensurate curve on the sphere can be obtained by translating and rotating this one.

\begin{figure}[htb]
\begin{center}
\leavevmode
\includegraphics[width=2in]{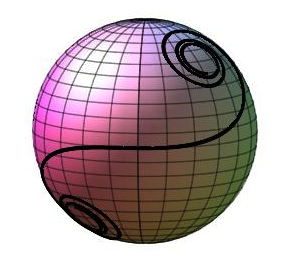} \hspace{0.5in}
\includegraphics[width=2in]{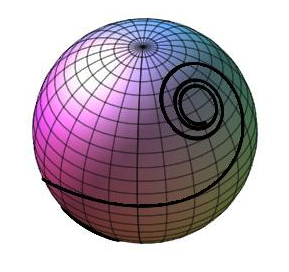}
\caption{Two views of a commensurate curve on the sphere}
\label{sphere-curves-fig}
\end{center}
\end{figure}

\begin{remark}
Because all normal curvatures on the unit sphere are equal to 1 and the commensurate curves have curvature function $\kappa(\bar{s}) = \sqrt{\bar{s}^2 + 1}$, they have {\em geodesic} curvature function $\kappa_g(\bar{s}) = \bar{s}$.  Thus the commensurate curves are the spherical analogs of plane curves with curvature $\kappa(\bar{s}) = \bar{s}$.  These plane curves are called {\em Euler spirals} or {\em clothoid} curves (see, e.g., \cite{AS64}), and they have a long and interesting history.  They first appeared as the solution to an elasticity problem posed in 1694 by James Bernoulli \cite{Bernoulli64}, then in work of Augustin Fresnel in 1816 regarding the problem of light diffracting through a slit \cite{Fresnel16}, and again in work of Arthur Talbot in 1901 related to designing railroad tracks so as to provide as smooth a riding experience as possible \cite{Talbott01}.  A nice account of the history of Euler spirals is given in \cite{Levien08}.
\end{remark}

\end{example}

For the remaining examples in this section, we computed commensurate curves as follows.  Example \ref{cool-curves-on-sphere-ex} illustrates how the local parametrization \eqref{local-param-assumption} may be too limiting, as it assumes that the curve is never tangent to the $v$-parameter curves of $\Sigma$ and so may only be accurate for computing small segments of the curve.  In order to remedy this weakness, for the remaining examples we assume that the curve is parametrized as
\[ \alp(t) = X(u(t), v(t)), \]
where
\begin{equation}
 u'(t) = \cos(\theta(t)), \qquad v'(t) = \sin(\theta(t)) \label{part-of-ode-system}
\end{equation}
for some unknown function $\theta(t)$.  Then, for a given parametrization 
$X(u,v)$ of $\Sigma$, equation \eqref{big-theorem-eqn} becomes a second-order ODE for the function $\theta(t)$, with coefficients depending on the functions $u(t), v(t)$.  We used the Rosenbrock stiff algorithm in {\sc Maple} 15 to numerically solve the system consisting of this ODE together with equations \eqref{part-of-ode-system} for various choices of initial conditions in order to generate the curves in the following examples.

\begin{example}
Let $\Sigma \subset \R^3$ be a paraboloid, parametrized via a Monge patch in polar coordinates as
\[ X(u,v)  = \left[v\cos (u), v\sin (u),v^2  \right]. \] 
Some commensurate curves on $\Sigma$ are shown in Figure \ref{paraboloid-curves-fig}. 
\begin{figure}[htb]
\begin{center}
\leavevmode
\includegraphics[width=2in]{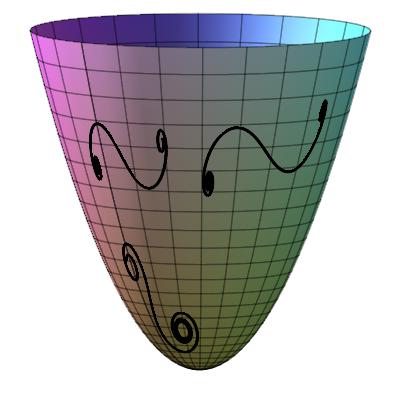} 
\caption{Commensurate curves on the paraboloid}
\label{paraboloid-curves-fig}
\end{center}
\end{figure}
As one might expect since the paraboloid is an elliptic surface, these curves appear qualitatively similar to the commensurate curves on the sphere.

\end{example}

In our next three examples, we compute some examples of nondegenerate commensurate curves on hyperbolic surfaces.  Such curves can never be tangent to an asymptotic direction, since these are precisely the null directions for $\text{I}_{\text{aff}}$.  (If an asymptotic curve happens to be contained in a plane, as is the case for any straight line contained in a surface, then the curve is technically commensurate, but then it is also degenerate.)  Experimentally, we observe in all three cases that commensurate curves tend to approach asymptotic tangent directions fairly quickly, and that the numerical integration algorithm breaks down when the tangent vector to the curve gets too close to an asymptotic direction.

\begin{example}
Let $\Sigma \subset \R^3$ be a hyperbolic paraboloid, parametrized as 
\[ X(u,v) = \left[u ,v, u v  \right], \] 
so that the coordinate curves are precisely the two families of straight lines in $\Sigma$.
Some commensurate curves on $\Sigma$ are shown in  Figure \ref{hyperbolic-paraboloid-curves-fig}.

\begin{figure}[htb]
\begin{center}
\leavevmode
\includegraphics[width=2in]{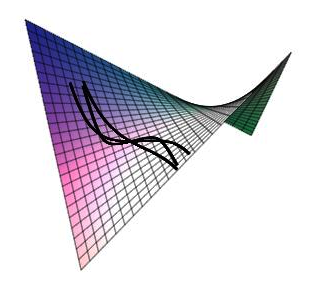}\caption{Commensurate curves on the hyperbolic paraboloid}
\label{hyperbolic-paraboloid-curves-fig}
\end{center}
\end{figure}

\end{example}

\begin{example}
Let $\Sigma \subset \R^3$ be a hyperboloid, parametrized as
 \[ X(u,v)  = \left[\cos(u)-v\sin(u),\sin(u)+v\cos(u),v \right],\] 
so that the $v$ coordinate curves are one of the two families of straight lines in $\Sigma$.
Some commensurate curves on $\Sigma$ are shown in  Figure \ref{hyperboloid-curves-fig}.

\begin{figure}[htb]
\begin{center}
\leavevmode
\includegraphics[width=2in]{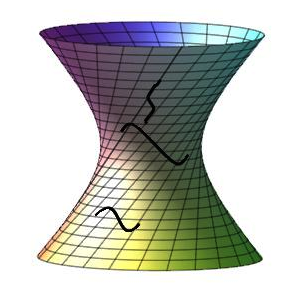} 
\caption{Commensurate curves on the hyperboloid}
\label{hyperboloid-curves-fig}
\end{center}
\end{figure}

\end{example}

\begin{example}
Let $\Sigma \subset \R^3$ be a helicoid, parameterized as
 \[ X(u,v)  = \left[u\cos(v), u\sin(v), v \right]. \] 
Some commensurate curves on $\Sigma$ are shown in  Figure \ref{helicoid-curves-fig}.

\begin{figure}[htb]
\begin{center}
\leavevmode
\includegraphics[width=2in]{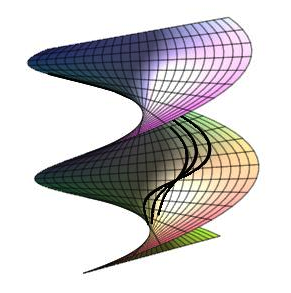} \hspace{0.5in}
\includegraphics[width=2in]{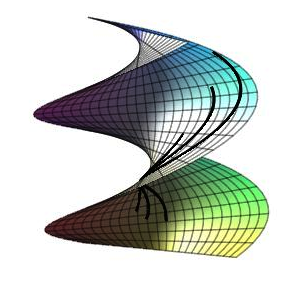} 
\caption{Commensurate curves on the helicoid}
\label{helicoid-curves-fig}
\end{center}
\end{figure}

\end{example}

% Here's how to include a picture; you don't have to include the file extension (.jpg or whatever)
%\begin{figure}[h]
%\centering
%\includegraphics[width=2in]{niftypicture}
%\caption{Isn't this a cool picture?}
%\label{niftypicture}
%\end{figure}

% If I wanted the bibliography to include references I haven't cited in the paper; here's the command for that:
%\nocite{ThisGuy}

% The bibliographic info (in BibTex format) should be contained in the file TaleOfTwoArcLengths-bib.bib

\bibliographystyle{amsplain}
\bibliography{TaleOfTwoArcLengths-bib}

\end{document}